\documentclass[reqno, a4paper]{amsart}
\usepackage{amssymb,amscd,amsthm,latexsym,qsymbols}
\usepackage[all]{xy}
\CompileMatrices

\usepackage[british]{babel}
\usepackage{amsmath}
\usepackage{amsfonts}
\usepackage{amssymb}
\usepackage{amsthm}
\usepackage[all]{xy}

\usepackage[latin1]{inputenc}

\newtheorem{theorem}{Theorem}[section]
\newtheorem{lemma}[theorem]{Lemma}

\newtheorem*{property-P}{Property P}

\theoremstyle{definition}

\newtheorem{definition}[theorem]{Definition}

\newtheorem{remark}[theorem]{Remark}

\newdir{|>-}{!/  4.5pt / @{ }*@{ }*@{ }*!/ 0pt / @{|}*!/ -4.5pt / :(1,-.2)@^{>}*!/ -4.5pt / :(1,+.2)@_{>}}

\newdir{-|>}{!/ 9pt / @{|}*!/ 4.5pt /  :(1,-.2)@^{>}*!/ 4.5pt /  :(1,+.2)@_{>}*@{ }*@{ }*@{ } }

\newdir{ >}{!/  8.5pt / @{ }*@{ }*@{ }*@{>}}

\renewcommand{\to}{\longrightarrow}

\newcommand{\coker}{\ensuremath{\mathsf{coker\,}}}

\renewcommand{\ker}{\ensuremath{\mathsf{ker\,}}}

\newcommand{\Fc}{\ensuremath{\mathcal{F}}}

\newcommand{\Bc}{\ensuremath{\mathcal{B}}}
\newcommand{\Ac}{\ensuremath{\mathcal{A}}}

\newcommand{\Ab}{\ensuremath{\mathsf{Ab}}}

\newcommand{\Xc}{\ensuremath{\mathcal{X}}}
\newcommand{\Yc}{\ensuremath{\mathcal{Y}}}

\newcommand{\Qc}{\ensuremath{\mathcal{Q}}}

\newcommand{\Sc}{\ensuremath{\mathcal{S}}}

\newcommand{\Set}{\ensuremath{\mathsf{Set}}}

\newcommand{\Loop}{\ensuremath{\mathsf{Loop}}}
\newcommand{\Gp}{\ensuremath{\mathsf{Grp}}}

\newcommand{\Ext}{\ensuremath{\mathsf{Ext}}}

\newcommand{\CExt}{\ensuremath{\mathsf{CExt}}}

\newcommand{\Ec}{\ensuremath{\mathcal{E}}}

\newcommand{\ttildef}{\ensuremath{\tilde f}}

\newbox\pullbackbox
\setbox\pullbackbox=\hbox{\xy 0;<1mm,0mm>: \POS(4,0)\ar@{-} (0,0) \ar@{-} (4,4)
\endxy}
\def\pullback{\copy\pullbackbox}
\newbox\pushoutbox
\setbox\pushoutbox=\hbox{\xy 0;<1mm,0mm>: \POS(0,4)\ar@{-} (0,0) \ar@{-} (4,4)
\endxy}
\def\pushout{\copy\pushoutbox}
\begin{document}

\title{A description of the fundamental group in terms of commutators and closure operators}

\author{Mathieu Duckerts-Antoine}
\address{Institut de Recherche en Math\'ematique et Physique, Universit\'e catholique de Louvain, Chemin du Cyclotron 2, 1348 Louvain-la-Neuve, Belgium}
\email{mathieu.duckerts@uclouvain.be}

\author{Tomas Everaert}
\address{Vakgroep Wiskunde, Vrije Universiteit Brussel, Pleinlaan 2, 1050 Brussel, Belgium}
\email{teveraer@vub.ac.be}

\author{Marino Gran}
\address{Institut de Recherche en Math\'ematique et Physique, Universit\'e catholique de Louvain, Chemin du Cyclotron 2, 1348 Louvain-la-Neuve, Belgium}
\email{marino.gran@uclouvain.be}

\begin{abstract}
A connection between the Galois-theoretic approach to semi-abelian homology and the homological closure operators is established. In particular, a generalised Hopf formula for homology is obtained, allowing the choice of a new kind of
 functors as coefficients. This makes it possible to calculate the fundamental groups corresponding to many interesting reflections arising, for instance, in the categories of groups, rings, compact groups and simplicial loops.  
\end{abstract}

\maketitle



\section{Introduction}
The main purpose of this article consists in establishing a new connection between the study of generalised Hopf formulae for semi-abelian homology \cite{EVD, EGV, GV, Ev} and the so-called homological closure operators which arise in the realm of homological categories \cite{BG}. This work provides a way of calculating the homology in interesting new contexts, even beyond the case where the coefficients are taken in a reflector $I \colon \Ac \rightarrow \Bc$ to a Birkhoff subcategory $\Bc$ of a semi-abelian category $\Ac$ \cite{JMT}.

The main concepts needed for this purpose are the one of \emph{abstract fundamental group} in the sense of categorical Galois theory \cite{J4}, and the one of \emph{protoadditive functor} \cite{EG}, which extends the classical notion of additive functor to a non-abelian setting.

Before introducing the general context we shall consider in this article, let us first recall some known facts concerning the relationship between the fundamental group and the second integral homology group ${H}_2(B, \mathbb Z)$ of a group $B$.
 By the well known Hopf formula \cite{H}, the group ${H}_2(B, \mathbb Z)$ can be calculated, from any free presentation 
\begin{equation}\label{exact}
\xymatrix{
0\ar[r]&K\ar@{{|>-}->}[r]  & P\ar@{{-|>}}[r] &B\ar[r]&0
}  
\end{equation}
of $B$,
 as the quotient group 
\begin{equation}\label{HopfFormula}
{H}_2(B, \mathbb Z) \cong \frac{ K \wedge [P,P]}{[K,P]},
\end{equation}
where $[\cdot, \cdot ]$ is the group commutator. From a categorical perspective \cite{J4}, this formula can be revisited as follows. Given a surjective homomorphism $p \colon E \rightarrow B$ in the category $\mathsf{Grp}$ of groups, the abelianisation functor $\mathsf{ab} \colon \mathsf{Grp} \rightarrow \mathsf{Ab}$ sends the kernel equivalence relation of $p$, pictured as
$$\xymatrix{
(E \times_B E) \times_E (E \times_B E) \ar[r]^-{\tau} & E\times_B E \ar@<1ex>[r]^-{\pi_1} \ar@<-1ex>[r]_-{\pi_2} \ar@(ul,ur)[]^-{\sigma} & E,\ar@{.>}[l]|-{\delta}\\
}$$
(with $\delta$ the arrow giving the reflexivity, $\sigma$ the symmetry and $\tau$ the transitivity) to an internal groupoid $$\xymatrix{
\mathsf{ab}((E \times_B E) \times_E (E \times_B E)) \ar[r]^-{\mathsf{ab} (\tau)}& \mathsf{ab}(E\times_B E) \ar@<1ex>[rr]^-{\mathsf{ab}(\pi_1)} \ar@<-1ex>[rr]_-{\mathsf{ab}(\pi_2)} \ar@(ul,ur)[]^-{\mathsf{ab}(\sigma)} & &\mathsf{ab}(E)\ar@{.>}[ll]|-{\mathsf{ab}(\delta)}\\
}$$
in the category $\mathsf{Ab}$ of abelian groups, which is called the \emph{Galois groupoid} of $p$, written $\mathsf{Gal}(E,p)$ \cite{J, BoJ}. The group of automorphisms of $0$ of $\mathsf{Gal}(E,p)$ is, by definition, the \emph{Galois group} of $p$: this object is defined categorically, as the domain of the kernel of the induced arrow $\langle \mathsf{ab} (\pi_1), \mathsf{ab} (\pi_2) \rangle \colon  \mathsf{ab}(E\times_B E) \rightarrow \mathsf{ab}(E) \times \mathsf{ab}{(E)}$. 

Now, if we begin with a free presentation of $B$ as in (\ref{exact}), we obtain a  \emph{weakly universal central extension} by considering the quotient $q$ giving its centralisation $\overline{p}$:
$$
\xymatrix{
P \ar[dr]_p  \ar[rr]^q &  & \frac{P}{[K,P]} \ar@{.>}[dl]^{\overline{p}} \\
& {B.}& }  
$$
The Galois group of the weakly universal central extension $\overline{p} \colon  \frac{P}{[K,P]} \rightarrow B$ turns out to be an \emph{invariant of $B$}, called the \emph{fundamental group $\pi_1(B)$ of $B$}, which is independent of the chosen weakly universal central extension of $B$, and is isomorphic to the quotient on the right hand side of (\ref{HopfFormula}):
\begin{equation}\label{GaloisHopf}
 \pi_1 (B) \cong \frac{ K \wedge [P,P]}{[K,P]}.
 \end{equation}

The fundamental group $\pi_1 (B)$ of an object $B$ can be defined and studied in many different situations, essentially when there is a ``good adjunction'' that induces an admissible Galois structure in the sense of \cite{J} (see Section \ref{GaloisSection}). For instance, the Poincar\'e fundamental group of homotopy classes of loops at a fixed base point is another special instance of this general notion of fundamental group, corresponding to a reflector of rather different nature: the connected component functor $\pi_0 \colon \mathsf{LoCo} \rightarrow \mathsf{Set}$ from the category of locally connected topological spaces to the category of sets. In this case, universal coverings play the role that weakly universal central extensions play above (see Chapter $6$ in \cite{BoJ}). This example motivates the use of the term fundamental group, and of the symbol $\pi_1(B)$, for the Galois group of an object $B$ in a more general context.

The isomorphisms (\ref{HopfFormula}) and (\ref{GaloisHopf}) can be extended to the context of a semi-abelian category $\Ac$ with enough regular projectives \cite{EVD, EGHHA, J4}.
 Here the coefficients are taken in any reflector $I \colon \Ac \rightarrow \Bc$ from
$\Ac$ to any Birkhoff subcategory $\Bc$ of $\Ac$ (i.e. $\Bc$ is a  full regular-epi reflective subcategory stable under regular quotients in $\Ac$), on the model of the reflector $\mathsf{ab} \colon \mathsf{Grp} \rightarrow \mathsf{Ab}$, and the commutator $[\cdot , \cdot ]$ is replaced by a ``relative'' commutator $[\cdot , \cdot ]_{\Bc}$. Once again, the fundamental group $\pi_1(B)$ of $B$ can be defined as the (internal) group of automorphisms of the (internal) Galois groupoid of \emph{any} weakly universal central extension of $B$. As in the case of the category of abelian groups, an object in $\Bc$ carries at most one group structure, so that the Galois group is uniquely determined by its underlying object. For instance, when $B$ is an object in a semi-abelian category $\Ac$ which is monadic over sets we obtain an isomorphism between the fundamental group $\pi_1(B)$ and the homology object $H_2 (B, I)$ of $B$ with coefficients in the reflector $I \colon \Ac \rightarrow \Bc$ (in the sense of Barr and Beck \cite{BarrBeck}). 

The main point of this article is to further extend the isomorphism (\ref{GaloisHopf}) to a more general situation, where the coefficient functor is not necessarily a reflector to a Birkhoff subcategory. More precisely, we examine the following composite adjunction
 \begin{equation}\label{Composite-adjunction1}
 \xymatrix@=30pt{
\Fc \ar@/_/[r]_-U   \ar@{}[r]|-{\perp} & \Bc \ar@/_/[r]_-H \ar@{}[r]|-{\perp} \ar@/ _/[l]_-F & \Ac \ar@/ _/[l]_-I}
\end{equation}
where $\Ac$ is a semi-abelian category, $\Bc$ a Birkhoff subcategory of $\Ac$, and $\Fc$ an admissible (in the sense of categorical Galois theory) regular epi-reflective subcategory of $\Bc$ with the property that the reflector $F \colon \Bc \rightarrow \Fc$ is protoadditive. This means that $F$ preserves split short exact sequences:
if 
$$\xymatrix{
0\ar[r]&K\ar@{{|>-}->}[r]^-{k}  & A\ar@{{-|>}}[r]_-f &B \ar@<-1ex>[l]_s \ar[r]&0
}$$
is a split short exact sequence, then its image 
$$\xymatrix{
0\ar[r]&F(K)\ar@{{|>-}->}[r]^-{F(k)}  & F(A)\ar@{{-|>}}[r]_-{F(f)} &F(B) \ar@<-1ex>[l]_{F(s)} \ar[r]&0
}$$
in $\Bc$ by $F$ is again a split short exact sequence. Of course, a functor between additive categories is additive if and only if it is protoadditive, but there are many further interesting examples of protoadditive functors between homological categories \cite{EG, Infinite}. 

Our main result is Theorem \ref{Formule}: it states that, given a projective presentation as in (\ref{exact}), the fundamental group of $B$ relative to the adjunction (\ref{Composite-adjunction1}) is given by
$$
\pi_1(B)\cong \dfrac{K `^ \overline{([P,P]_{\Bc})}^{\Fc}_P }{\overline{([K,P]_{\Bc})}^{\Fc}_{K}},
$$
where the closure is the \emph{homological closure} corresponding to the regular-epi reflector $F \circ I \colon \Ac \rightarrow \Fc$ (see Section \ref{closure.operator} for more details on this closure operator). There is a wide range of reflections for which it is possible to ``compute'' the fundamental group on any given object using this formula, which includes, in particular, the known one in the Birkhoff case, when $\Bc= \Fc$. Various examples are examined in detail in the last section, such as the reflector  $\Gp \rightarrow \Ab_{t.f.}$ from the category of groups to the category of torsion-free abelian groups and the reflector $\mathsf{Grp} (\mathsf{HComp}) \rightarrow \mathsf{Ab} (\mathsf{Prof})$ from the category of compact Hausdorff groups to the category of profinite abelian groups.


 \section{Homological and semi-abelian categories}\label{homological}

 
 We assume the reader to be familiar with the notions of homological and of semi-abelian categories \cite{JMT}: we briefly recall some definitions and properties below, and we refer to the book \cite{BB} for more details.
 
Recall that a category $\Ac$ is \emph{homological} when it is finitely complete, regular, pointed (with zero object $0$) and protomodular \cite{Bourn0}: in the presence of the other assumptions this last property amounts to the validity of the Split Short Five Lemma in $\Ac$.
It is well known that in a homological category $\Ac$ any regular epimorphism is a \emph{normal} epimorphism, thus the cokernel of its kernel \cite{B2}. We write 
$$\xymatrix{
0\ar[r]&K\ar@{{|>-}->}[r]^-{k}  & A\ar@{{-|>}}[r]^-f &B\ar[r]&0
}$$
for a \emph{short exact sequence}, by which we mean that $k =\ker (f)$ and $f = \coker (k)$. We shall also write $\xymatrix{K \ar@{{|>-}->}[r] & A}$ for a normal monomorphism and  $\xymatrix{A\ar@{{-|>}}[r] & B}$ for a normal epimorphism.
The following result is well known (see \cite{B2}); we shall often use it in the paper.
 \begin{lemma}\label{homologique}
Let $\Ac$ be a homological category and consider the following commutative diagram:
\begin{equation}\label{2-by-3}
\xymatrix{
0\ar[r]&K\ar@{{|>-}->}[r]^-{k}\ar[d]_-u \ar@{}[rd]|-{(1)}& A\ar@{{-|>}}[r]^-f\ar[d]_-v \ar@{}[rd]|-{(2)}&B\ar[r]\ar[d]^-w&0\\
0\ar[r]&K'\ar@{{|>-}->}[r]_-{k'}& A'\ar@{{-|>}}[r]_-{f'}& B'\ar[r]&0
}
\end{equation}
where both rows are exact. Then

\begin{itemize}
\item $w$ is a monomorphism if and only if $(1)$ is a pullback.\\
 \item $u$ is an isomorphism if and only if $(2)$ is a pullback.\\
\end{itemize}
\end{lemma}
 
 A homological category $\Ac$ is a \emph{semi-abelian category} when it is Barr-exact \cite{Barr} (any internal equivalence relation is effective) and has binary coproducts.
An additional property of a semi-abelian category is that the regular image of a normal monomorphism is again normal: given a commutative square
$$\xymatrix{
K \ar@{{-|>}}[r]^-p \ar@{{|>-}->}[d]_-k & f(K) \ar@{{ >}->}[d]^-m \\
A \ar@{{-|>}}[r]_-f & B
}$$
where $p$ and $f$ are regular epimorphisms, $k$ is a normal monomorphism and $m$ a monomorphism, then $m$ is necessarily normal \cite{JMT}. 

It is well known that any full regular-epi reflective subcategory of a homological category is itself homological, whereas a Birkhoff subcategory of a semi-abelian category is semi-abelian \cite{BG1}.


 \section{Galois structures and normal extensions}\label{GaloisSection}

 
 We now recall the definitions of a Galois structure and of a normal extension. We shall restrict ourselves to the case of reflective subcategories \cite{JK,JK4}.
 \begin{definition}\label{Structure-Galois} \cite{J} A {\it Galois structure} is a system $\Gamma =(\Ac,\Bc,I,H,\eta,\epsilon,\Ec)$, where:\\
 \begin{enumerate}
 \item 
 $\Bc$ is a full replete reflective subcategory of $\Ac$ with inclusion functor $H$ and left adjoint $I$, unit $\eta$ and counit $\epsilon$ (which is an isomorphism, of course)
 $$\xymatrix@=30pt{
{\Bc \, } \ar@/_/[r]_-{H} \ar@{}[r]|-{\perp} & {\, \Ac \, },
\ar@/_/[l]_-{I} }$$

\item $\Ec$ is a class of morphisms in $\Ac$, such that:\\
\begin{enumerate}
\item $HI(\Ec) \subseteq \Ec$;
\item $\Ac$ has all pullbacks along morphisms in $\Ec$;
\item $\Ec$ is closed under composition, contains all isomorphisms, and is pullback-stable along morphisms in $\Ac$.
\end{enumerate}
\end{enumerate}
\end{definition}
 
 Such a Galois structure induces an adjunction $$(I^B,H^B,\eta^B,\epsilon^B): \Ac\downarrow B \rightharpoonup \Bc\downarrow I(B)$$ for every $B$ in $\Ac$,  where\\
 \begin{enumerate}
 \item $\Ac\downarrow B$ is the full subcategory of the slice category $\Ac/B$ whose objects, the \emph{extensions of $B$}, are the arrows in $\Ec$ with codomain $B$, and $\Bc\downarrow I(B)$ is the full subcategory of $\Bc/I(B)$ whose objects are in $\Ec$; we write $(A,f)$ or $(A, f \colon A \rightarrow B)$ to denote an extension of $B$;
 \item $I^B(A,f:A\rightarrow B)=(I(A),I(f):I(A)\rightarrow I(B))$; 
 \item $H^B(X,\phi:X \rightarrow I(B))$ is defined as the first projection $\pi_1$ of the following pullback:
 $$\xymatrix{
    B{\times}_{HI(B)}H(X)  \ar[rr]^{\pi_2} \ar[d]_{\pi_1} \ar@{}[rd]|<<{\pullback} & & H(X) \ar[d]^{H(\phi)} \\
    B \ar[rr]_{\eta_B}& & HI(B);
  }$$
 \item $(\eta^B)_{(A,f)}=\langle f,\eta_A \rangle:A\rightarrow B {\times}_{HI(B)}HI(A)$;
 \item $(\epsilon^B)_{(X,\phi)}= \epsilon_X \circ I(\pi_2):I(B{\times}_{HI(B)}H(X)) \rightarrow IH(X)\rightarrow X$.
 \end{enumerate} 
 
 \begin{definition}\label{Structure-Admissible} A Galois structure is {\it admissible} when, for any object $B$ in $\Ac$, $\epsilon^B$ is an isomorphism.
 \end{definition}
 Of course, the admissibility of the Galois structure amounts to the fully faithfulness of the functor $H^B \colon \Bc\downarrow I(B) \rightarrow  \Ac\downarrow B$, for every $B\in\Ac$.
 
 As shown by Janelidze and Kelly, when $\Ec$ is the class of regular epimorphisms, any Birkhoff subcategory of an exact category is admissible (= determines an admissible Galois structure), provided the lattice of congruences on any object in $\Ac$ is modular \cite{JK}. This is the case, in particular, for any Birkhoff subcategory of a semi-abelian category \cite{BG1}. From now on, we shall assume that $\Gamma$ denotes an admissible Galois structure. 
 
Given a Galois structure $\Gamma$, the purpose of categorical Galois theory is then to describe and classify the morphisms in $\Ec$ which are $\Gamma$-\emph{coverings}, a notion that we are going to recall below.
 
\begin{definition}\label{Extension-Mon} A morphism $p:E \rightarrow B$ in $\Ec$ is a {\it  monadic extension} if the pullback functor $p^*:\Ac\downarrow B \rightarrow \Ac\downarrow E$ is monadic.
 \end{definition}
 \begin{remark}
 For the main results in this article, we shall always assume that $\Ec$ is the class of all regular epimorphisms in a Barr-exact category. It is well known that, in this context, a regular epimorphism is always a monadic extension. 
\end{remark}
\begin{definition}\label{Trivial-extension} A morphism $f \colon E \rightarrow B$ in $\Ec$ is said to be a {\it  $\Gamma$-trivial extension}, or a {\it $\Gamma$-trivial covering}, if the following commutative square is a pullback:
$$\xymatrix{
E \ar[r]^-{\eta_E}\ar[d]_-f&HI(E)\ar[d]^-{HI(f)}\\
B\ar[r]_-{\eta_B} &HI(B)
}$$
\end{definition}
Notice that a morphism  $f \colon E \rightarrow B$ in $\Ec$ is a trivial extension precisely when it lies in the replete image of the fully faithful functor $H^B \colon \Bc\downarrow I(B) \rightarrow  \Ac\downarrow B$.
\begin{definition}\label{Central-extension} A morphism $f$ in $\Ec$ is a {\it  $\Gamma$-central extension}, or a {\it $\Gamma$-covering}, if there is  a monadic extension $g$ such that $g^*(f)$ is a trivial extension.
\end{definition}

\begin{definition}\label{Normal-extension} A monadic extension $f$ is a {\it  $\Gamma$-normal extension} if $f^*(f)$ is a trivial extension.
\end{definition}
We shall sometimes speak of $\Bc$-trivial, $\Bc$-central and $\Bc$-normal extensions if the Galois structure $\Gamma$ is understood.

Recall from \cite{JK} that if $\Bc$ is a Birkhoff subcategory of an exact Goursat category (for instance, of a semi-abelian category) $\Ac$, and $\Ec$ is the class of regular epimorphisms in $\Ac$, then the $\Bc$-central and $\Bc$-normal extensions coincide. The same will be true for the Galois structures considered below (see Theorem \ref{Caracterisation}). Note also that when $\Ac$ is the variety of groups and $\Bc$ the subvariety of abelian groups, then a normal or central extension of groups is just a central extension in the usual sense: a surjective homomorphism $f\colon A\to B$ whose kernel lies in the center of $A$ (see, for instance \cite{BoJ}).



\section{The fundamental group}\label{Pi1}


Let $\Gamma=(\Ac,\Bc,I,H,\eta,\epsilon, \Ec)$ be an admissible Galois structure such that $\Ac$ is a homological category. With any normal extension $(E,p:E \rightarrow B)$ we can associate an internal groupoid  $\mathsf{Gal}(E,p)$, called its {\it Galois groupoid}, and an internal group $\mathsf{Gal} (E,p,0)$, its {\it Galois group}. If the diagram
$$\xymatrix{
(E \times_B E) \times_E (E \times_B E) \ar[r]^-{\tau} & E\times_B E \ar@<1ex>[r]^-{\pi_1} \ar@<-1ex>[r]_-{\pi_2} \ar@(ul,ur)[]^-{\sigma} & E\ar@{.>}[l]|-{\delta}\\
}$$
represents the kernel equivalence relation of $p$ (viewed as an internal groupoid in $\Ac$), then:
\begin{enumerate}
\item the Galois groupoid $\mathsf{Gal}(E,p)$ is the image under $I$ of the kernel equivalence relation of $p$:
$$\xymatrix{
I((E \times_B E) \times_E (E \times_B E)) \ar[r]^-{I(\tau)}& I(E\times_B E) \ar@<1ex>[rr]^-{I(\pi_1)} \ar@<-1ex>[rr]_-{I(\pi_2)} \ar@(ul,ur)[]^-{I(\sigma)} & &I(E)\ar@{.>}[ll]|-{I(\delta)}\\
}.$$
One can prove that there is an isomorphism $$I(E\times_B E)\times_{I(E)} I(E\times_B E) \cong I((E \times_B E) \times_E (E \times_B E))$$ and that $\mathsf{Gal}(E,p)$ defined this way is an internal groupoid in $\Bc$ (see \cite{BoJ}).\\
\item The Galois group is defined as the object $\mathsf{Gal} (E,p,0)$ in the following pullback:
$$\xymatrix@=40pt{
\mathsf{Gal} (E,p,0) 
\ar@{}[rd]|<<<{\pullback} \ar[r] \ar@{{ >}->}[d]&0\ar@{{ >}->}[d]\\
I(E\times_B E)\ar[r]_-{\langle  I(\pi_1),I(\pi_2) \rangle}& I(E)\times I(E).
}$$
\end{enumerate}
The Galois group $\mathsf{Gal} (E,p,0)$  can be viewed, internally, as the group of automorphisms of $0$. Since $\Bc$ is a protomodular category, any of its objects underlies at most one internal group structure \cite{BB}. As explained in \cite{J4}, the Galois group $\mathsf{Gal} (E,p,0)$ ``measures'' the lack of preservation of the pullback $E\times_B E$ by the functor $I$. 

Recall that a normal extension $(E, p\colon E\to B)$ is called \emph{weakly universal} if, given any other normal extension $(E', p'\colon E'\to B)$, there is a morphism $u \colon E \rightarrow E'$ such that $p=p' \circ u$. In the article \cite{J4} the author proved that the Galois group $\mathsf{Gal} (E,p,0)$ of a weakly universal normal extension $p\colon E\to B$ is an invariant of $B$, denoted by $\pi^{\Gamma}_1(B)$, or just $\pi_1(B)$, the \emph{fundamental group} of $B$. Furthermore, in Theorem $2.1$ it was shown that $\pi_1(B)$ is isomorphic to the following intersection
\begin{equation}\label{Hopf} 
\pi_1(B) \cong K[p] `^ K[\eta_E].
\end{equation}
where $(E,p\colon E\to B)$ is any weakly universal normal extension of $B$, and $K[p] `^  K[\eta_E]$ denotes the domain of the intersection of the kernels $\ker(p)$ and $\ker(\eta_E)$ of $p$ and $\eta_E$, respectively.



\section{Homological closure operators}\label{closure.operator}

 
In \cite{BG} the notion of \emph{homological closure operator} was introduced and a bijective correspondence was established between such closure operators and regular epi-reflective subcategories of a given homological category. 

\begin{definition}
A homological closure operator on normal subobjects associates, with any normal subobject $k \colon K \rightarrow  A$ of $A$ in a homological category $\Ac$, another normal subobject $\overline{k}\colon \overline{K}_A \rightarrow A $, the closure of $K$ in $A$. This assignment has to satisfy the following properties, where $K \rightarrow A$ and  $L \rightarrow A $ are normal subobjects, and $f \colon B \rightarrow A$ is an arrow in $\Ac$:
\begin{enumerate}
\item $K \subseteq \overline{K}_A$,
\item $K \subseteq L$ implies $\overline{K}_A \subseteq \overline{L}_A$,
\item $\overline{(f^{-1} (K))}_B \subseteq  f^{-1} ( \overline{K}_A)$,
\item $\overline{\overline{K}}_A = \overline{K}_A$,
\item for any regular epimorphism $g \colon B \rightarrow A$ we have $\overline{(g^{-1} (K))}_B =  g^{-1} ( \overline{K}_A)$.
\end{enumerate}
\end{definition}
The bijection between regular epi-reflective subcategories of $\Ac$ and homological closure operators in $\Ac$ is established as follows. Given a homological closure operator, the corresponding regular epi-reflective subcategory $\Bc$ of $\Ac$ is its full replete subcategory whose objects $B$ have the property that $0 \rightarrow B$  is closed, and the reflection of an object $A \in \Ac$ into $\Bc$ is given by the quotient $A/{\overline 0_A}$ of $A$ by the closure ${\overline 0_A}$ of $0$ in $A$.  Conversely, given a regular epi-reflective subcategory $\Bc$ of $\Ac$, the closure $\overline{k}:\overline{K}_A^{\Bc}  \rightarrow A$ of a normal subobject $k:K \rightarrow A$ is given by the pullback 
$$\xymatrix{
{\overline K}_A^{\Bc} \ar@{}[rd]|<{\pullback} \ar@{{|>-}->}[d]_-{\overline k} \ar@{{-|>}}[r] & K[\eta_{A/K}] \ar@{{|>-}->}[d]^{\ker(\eta_{A/K})} \\
A \ar@{{-|>}}[r]_-{q_K} & A/K
}$$
where $q_K \colon A \rightarrow A/K$ is the canonical quotient and $\ker(\eta_{A/K})$ is the kernel of the unit of the reflection $\eta_{A/K}$ of $A/K$ into $\Bc$. 

%
Recall that, in a homological category $\Ac$, two normal subobjects $K\to A$ and $L\to A$ admit a supremum (in the poset of normal subobjects of $A$) as soon as the following pushout exists
 \[
 \xymatrix{
 A 
\ar@{{-|>}}[r] \ar@{{-|>}}[d] & A/K \ar@{{-|>}}[d]_>>{\pushout}
\\
 A/L \ar@{{-|>}}[r] & P,}
 \]
and that this supremum can be obtained as the kernel of the ``diagonal" of this square, the arrow $A\to P\cong A/(K\vee L)$.

Proposition $3.3$ in \cite{BG} gives the explicit formula $\overline{K}_A^{\Bc} =K `V {\overline 0}^{\Bc}_A$ to compute the closure $\overline{K}_A^{\Bc}$ of a normal subobject $K \rightarrow A$ in a semi-abelian category, whenever $\Bc$ is a Birkhoff subcategory of $\Ac$. 
Below we refine this observation, in the more general context of homological categories with pushouts of regular epimorphisms (we write $K\lhd A$ to indicate that $K$ is a normal subobject of $A$):
\begin{lemma}\label{Fermeture}
Let $\Bc$ be a regular epi-reflective subcategory of a homological category $\Ac$ such that the supremum of two normal subobjects always exists. The following properties hold:
\begin{enumerate} 
\item $\forall A \in \Ac$, $\forall K \lhd A$ one has ${\overline K}^{\Bc}_A= \overline {(K `V {\overline 0}^{\Bc}_A)}^{\Bc}_A$;
\item $\Bc$ is a Birkhoff subcategory if and only if $\forall A\in\Ac$, $\forall K \lhd A$ one has $$K `V {\overline 0}^{\Bc}_A = {\overline K}^{\Bc}_A;$$
\item if $\Bc$ is a Birkhoff subcategory, then  $\forall A\in\Ac$, $\forall K,L \lhd A$ one has 
\[
\overline{ (K `V L)}^{\Bc}_A=\overline{K}^{\Bc}_A `V \overline{L}^{\Bc}_A.
\]
\end{enumerate}
\end{lemma}

\begin{proof}
$(1)$ On the one hand we have $K \leq K `V {\overline 0}^{\Bc}_A$, which implies $\overline{K}^{\Bc}_A \leq \overline {(K `V {\overline 0}^{\Bc}_A)}^{\Bc}_A$.  On the other hand, since $K `V {\overline 0}^{\Bc}_A \leq \overline{K}^{\Bc}_A$, we find that $$\overline {(K `V {\overline 0}^{\Bc}_A)}^{\Bc}_A \leq \overline{(\overline{K}^{\Bc}_A)}^{\Bc}_A = \overline{K}^{\Bc}_A.$$

$(2)$ By the previous property it suffices to prove that $\Bc$ is a Birkhoff subcategory of $\Ac$ if and only if $K `V {\overline 0}^{\Bc}_A$ is closed in $A$. Let us first of all remark that for any short exact sequence 
\[
\xymatrix{ 0 \ar[r] & K \ar[r] & A \ar[r] & B \ar[r] & 0}
\]
in $\Ac$ we have that $K$ is closed in $A$ if and only $B$ lies in $\Bc$. Now, if $\Bc$ is a Birkhoff subcategory of $\Ac$ with reflector $I \colon \Ac \rightarrow \Bc$, then $K `V \overline{0}^{\Bc}_A$ is closed in $A$ since the quotient $A/K `V \overline{0}^{\Bc}_A$ lies in $\Bc$ as a regular quotient of $A/\overline{0}^{\Bc}_A=I(A)\in\Bc$.  Conversely, if we assume that $K `V \overline{0}^{\Bc}_A$ is closed in $A$, for any short exact sequence as above, then if $A$ lies in $\Bc$, so that $\overline{0}^{\Bc}_A=0$, we have that $K$ is closed in $A$, hence $B\in\Bc$.

$(3)$ If $\Bc$ is a Birkhoff subcategory of $\Ac$ one has the following equalities:
$$\begin{array}{rcl}
 \overline{ (K `V L)}^{\Bc}_A &
 \stackrel{(2)}{=}& (K `V L) `V {\overline 0}^{\Bc}_A\\
 &=& (K `V {\overline 0}^{\Bc}_A) `V (L `V {\overline 0}^{\Bc}_A)\\
 &\stackrel{(2)}{=}&\overline{K}^{\Bc}_A `V \overline{L}^{\Bc}_A.
\end{array}$$
\end{proof} 


\section{A composite adjunction}\label{Section-principale}


From now on we shall consider the following adjunctions
 \begin{equation}\label{Composite-adjunction}
 \xymatrix@=30pt{
\Fc \ar@/_/[r]_-U   \ar@{}[r]|-{\perp} & \Bc \ar@/_/[r]_-H \ar@{}[r]|-{\perp} \ar@/ _/[l]_-F & \Ac \ar@/ _/[l]_-I}
\end{equation}
where $\Ac$ is a semi-abelian category, $\Bc$ a Birkhoff subcategory of $\Ac$, and $\Fc$ an admissible (for the class of regular epimorphisms in $\Bc$) regular epi-reflective subcategory of $\Bc$ with the property that the reflector $F \colon \Bc \rightarrow \Fc$ is \emph{protoadditive} \cite{EG}: this means that $F$ preserves split short exact sequences.\\
The functor $F\circ I$ is left adjoint to $H \circ U$. We shall write $\eta$ and $\epsilon$ for the unit and the counit of this composite adjunction, and $\Ec$ for the class of regular epimorphisms in $\Ac$.
The unit and the counit of the adjunction $I\dashv H$ will be denoted by $\eta^1$ and $\epsilon^1$, whereas the unit and the counit of the adjunction $F\dashv U$ will be denoted  $\eta^2$ and $\epsilon^2$, respectively.  

\begin{remark}
The protoadditivity of the reflector $F$ is independent of the admissibility. For instance, the reflection $\Gp \to \Ab$ of the variety $\Gp$ of groups into the subvariety $\Ab$ of abelian groups is admissible, but not protoadditive. The reflection  $\Ab \to \Qc$ of the variety $\Ab$ into the quasi-variety $\Qc$ of abelian groups satisfying the implication ($4x = 0 \Rightarrow 2x=0$) is (proto)additive but it is not admissible (see \cite{Infinite}).  
\end{remark}

\begin{lemma}\label{Lemme-Admissible} $(\Ac, \Fc, F\circ I, H \circ U, \eta, \epsilon ,\Ec)$ is an admissible Galois structure.
\end{lemma}

\begin{proof}
Clearly, $(\Ac, \Fc, F\circ I, H \circ U, \eta, \epsilon ,\Ec)$ is a Galois structure. To see that it is admissible, note that, for any $B$ in $\Ac$, the functor $(H\circ U)^{B}\colon \Fc\downarrow FI(B)\to \Ac\downarrow B$ can be decomposed into 
\[
(H\circ U)^{B}=H^{B}\circ U^{I(B)}\colon  \Fc\downarrow FI(B) \to \Bc\downarrow I(B) \to \Ac\downarrow B,
\]
so that each $(H\circ U)^{B}$ is fully faithful as a composite of fully faithful functors.
\end{proof}

\begin{theorem}\label{Caracterisation} Let $f: A \rightarrow B$ be a regular epimorphism in $\Ac$. The following conditions are equivalent:
\begin{enumerate}
 \item $f$ is an $\Fc$-normal extension;
 \item $f$ is an $\Fc$-central extension;
 \item $f$ is $\Bc$-central and $K[f] \in \Fc$.
\end{enumerate}
\end{theorem}

\begin{proof}
$(1)\Rightarrow (2)$ Obvious.\\
$(2) \Rightarrow (3)$ 
Let us assume that $f$ is a $\Fc$-central extension: there exists a monadic extension $p:E \rightarrow B$ such that $p^*(f)$ is $\Fc$-trivial. Then in the following commutative diagram the composite of the left pointing squares is a pullback: 
\[
\xymatrix{
FI(E\times_BA) \ar[d]_{FI(p^*(f))} & I(E\times_BA) \ar[d]_{I(p^*(f))} \ar[l] & E\times_B A \ar@{}[rd]|<<{\pullback}\ar[d]_{p^*(f)} \ar[l] \ar[r] & A \ar[d]^f\\
FI(E) & I(E) \ar[l] & E \ar[l] \ar[r]_p & B}
\]
Since the middle square is a double extension (= a pushout of regular epimorphisms, in our context) because $\Bc$ is a Birkhoff subcategory of $\Ac$, this implies that this square is, in fact, a pullback (see Lemma $1.1$ in \cite{Gran}), and we find that $f$ is a central extension with respect to $\Bc$. 

That $K[f]$ lies in $\Fc$ for any $\Fc$-central extension $f$ is a very general and well-known fact. It suffices to note that the above pullbacks induce the isomorphisms 
\[
K[FI(p^*(f))] \cong K[p^*(f)] \cong K[f],
\]
which imply that $K[f]\in\Fc$, as (the domain of) the kernel of a morphism in $\Fc$.

$(3) \Rightarrow (1)$ Suppose now that $f$ is $\Bc$-central and $K[f] \in \Fc$. Then one sees in the diagram 
$$
\xymatrix@=35pt{
K[f] \ar@{{|>-}->}[r]^-{\ker(\pi_1)}  \ar@{{-|>}}[d]_-{\eta^1_{K[f]}} &R[f] \ar@<1ex>@{{-|>}}[r]^-{\pi_1}  \ar@<-1ex>@{{-|>}}[r]_-{\pi_2}  \ar@{{-|>}}[d]_-{\eta^1_{R[f]}}  \ar@{}[rd]|-{(3)} &A  \ar@{{-|>}}[d]^-{\eta^1_A}  \ar@{.>}[l]\\
I(K[f]) \ar[r]^-{I(\ker(\pi_1))}  \ar@{{-|>}}[d]_{\eta^2_{I(K[f])}} &I(R[f]) \ar@<1ex>@{{-|>}}[r]^-{I(\pi_1)}  \ar@<-1ex>@{{-|>}}[r]_-{I(\pi_2)}  \ar@{{-|>}}[d]  \ar@{}[rd]|-{(4)} &I(A) \ar@{{-|>}}[d]  \ar@{.>}[l] \\
FI(K[f]) \ar[r]_-{FI(\ker(\pi_1))} &FI(R[f]) \ar@<1ex>@{{-|>}}[r]^-{FI(\pi_1)}   \ar@<-1ex>@{{-|>}}[r]_-{FI(\pi_2)}  &FI(A) \ar@{.>}[l]} 
$$
 that $(3)$ is a pullback (since $f$ is $\Bc$-normal if and only if $f$ is $\Bc$-central by Theorem $4.8$ in \cite{JK}) and $\eta^1_{K[f]}$ an isomorphism. It follows that the (upper) second row is a split short exact sequence. By protoadditivity of $F$, also the (upper) third row is a split short exact sequence, and we obtain from Lemma \ref{homologique} that $(4)$ is a pullback because $\eta^2_{I(K[f])}$ is an isomorphism, by assumption. Thus $(3)$+$(4)$ is a pullback and $f$ is an $\Fc$-normal extension.
\end{proof}

Any Birkhoff subcategory $\Bc$ of a semi-abelian category $\Ac$ induces the reflection
$$\xymatrix{
\CExt_{\Bc}(\Ac) \ar@/_/[r]_-{H_1} \ar@{}[r]|-{\perp} & \Ext(\Ac) \ar@/_/[l]_-{I_1}
} $$
where $\Ext(\Ac)$ is the category of extensions in $\Ac$, and $\CExt_{\Bc}(\Ac)$ its full replete subcategory determined by the $\Bc$-central extensions in $\Ac$ (see \cite{EGV}). We recall that this reflection is defined by 
\[
I_1(A,f:A\rightarrow B) = (\frac{A}{[K[f],A]_{\Bc}}, \check f : \frac{A}{[K[f],A]_{\Bc}} \rightarrow B)
\]
where $\check{f}$ is the factorisation induced by the quotient $A\rightarrow A/[K[f],A]_{\Bc}$, and the ``relative commutator" $[K[f],A]_{\Bc}$ is obtained as the kernel of the restriction $\hat{\pi}_1$ of the first projection $\pi_1$ of the kernel pair of $f$ to $\overline{0}_{R[f]}^{\Bc}\to\overline{0}_A^{\Bc}$, as in the following diagram:
\[
\xymatrix{
0 \ar[r] & [K[f],A]_{\Bc} \ar[d] \ar[r]^-{\ker (\hat{\pi}_1)} \ar@{}[rd]|<<{\pullback} & \overline{0}^{\Bc}_{R[f]} \ar[d]_{\ker( \eta^1_{R[f]})}  \ar@<0.7ex>[r]^-{\hat{\pi}_1}  \ar@<-0.7ex>[r]_-{\hat{\pi}_2} \ar[d] & \overline{0}^{\Bc}_A \ar[d]^{\ker(\eta^1_A)} \ar[r] & 0\\
0 \ar[r] & K[f] \ar[r]_{\ker(\pi_1)} & R[f] \ar@<0.7ex>[r]^-{\pi_1}  \ar@<-0.7ex>[r]_-{\pi_2} & A \ar[r] & 0.}
\]
$[K[f],A]_{\Bc}$ is a normal subobject of $A$ via the monomorphism $\ker(\eta^1_A)\circ \hat{\pi}_2\circ \ker(\hat{\pi}_1)$ which is normal since it is the regular image of the normal monomorphism $\ker(\eta^1_{R[f]}) \circ \ker(\hat{\pi}_1)$ along $\pi_2$ (see Section \ref{homological}). It turns out that an extension $(A, f \colon A \rightarrow B)$ belongs to $\CExt_{\Bc}(\Ac)$ if and only if $[K[f], A ]_{\Bc} = 0$. Furthermore, the commutator $[K[\cdot],\cdot]_{\Bc}$ is stable under regular images in the following sense: if in diagram \eqref{2-by-3} both $u$ and $v$ are regular epimorphisms, then so is the induced morphism $[K,A]_{\Bc}\to [K',A']_{\Bc}$ (see \cite{EVDGT}, where the notation ``$L_1$" was used for the relative commutator). We recall that $[A,A]_{\Bc}=\overline{0}_A^{\Bc}$ for any $A$ in $\Ac$, since the reflector $I$ preserves binary products (see Lemma $5.2$ in \cite{EGHHA}).

 Below we are going to show that there is also an adjunction
$$\xymatrix{
\CExt_{\Fc}(\Ac) \ar@/_/[r]_-{U_1} \ar@{}[r]|-{\perp} & \CExt_{\Bc}(\Ac) \ar@/_/[l]_-{F_1}
} $$
induced by the reflector $F \colon \Bc \rightarrow \Fc$. This will show that any composite reflection (\ref{Composite-adjunction}) induces a composite reflection at the level of the category of extensions:
$$\xymatrix{
\CExt_{\Fc}(\Ac) \ar@/_/[r]_-{H_1 \circ U_1} \ar@{}[r]|-{\perp} & {\Ext(\Ac).} \ar@/_/[l]_-{F_1 \circ I_1}
} $$
For this, the following result will be useful:
\begin{lemma}\label{TKerF} 
Let $f:A\rightarrow B$ be a morphism in $\CExt_{\Bc}(\Ac)$. One has that \\ $\ker (f) \circ \ker(\eta_{K[f]}) \colon \overline{0}_{K[f]}^{\Fc}\rightarrow K[f]  \rightarrow A$ is a normal monomorphism: $\overline{0}_{K[f]}^{\Fc} \lhd  A $.
\end{lemma}

\begin{proof}
First note that since $f$ is a $\Bc$-central extension, its kernel $K[f]$ lies in $\Bc$ (see the proof of  $(2) \Rightarrow (3)$ in Theorem \ref{Caracterisation}). Then, as in the proof of $(3) \Rightarrow (1)$ in Theorem \ref{Caracterisation}, we see that (the upper part of) the last row in the following diagram is exact:
$$\xymatrix@=20pt{
&0 \ar[d]&&0\ar[d]&0\ar[d]\\
0\ar[r] &\overline{0}_{K[f]}^{\Fc} \ar@{}[rrd]|{(5)} \ar@{{|>-}->}[rr]^-{\widehat{\ker(\pi_1)}} \ar@{{|>-}->}[d]_{\ker(\eta_{K[f]})} && {\overline 0}^{\Fc}_{R[f]} \ar@<1ex>@{{-|>}}[r]^{\widehat{\pi}_1} \ar@<-1ex>@{{-|>}}[r]_{\widehat{\pi}_2} \ar@{{|>-}->}[d] & {\overline 0}^{\Fc}_{A} \ar@{{|>-}->}[d]^-{\ker(\eta_A)} \ar@{.>}[l] \ar[r]&0\\
0\ar[r]&K[f]  \ar[d]_{\eta_{K[f]}} \ar@{{|>-}->}[rr]_-{\ker(\pi_1)}&& R[f] \ar@<1ex>@{{-|>}}[r]^{\pi_1} \ar@<-1ex>@{{-|>}}[r]_{\pi_2}  \ar@{{-|>}}[d] & A \ar@{.>}[l]\ar@{{-|>}}[d] \ar[r]& 0\\
0\ar[r]&FI(K[f])\ar[d] \ar@{{|>-}->}[rr]_{FI(\ker(\pi_1))} && FI(R[f]) \ar[d]\ar@<1ex>@{{-|>}}[r]^{FI(\pi_1)} \ar@<-1ex>@{{-|>}}[r]_{FI(\pi_2)} & FI(A)\ar[d] \ar@{.>}[l] \ar[r] &0\\
&0&&0&0
}$$
It follows that the upper part of the first row in this diagram is also exact.
The commutative square $(5)$ is then a pullback, since $\ker(\eta_A)$ is a monomorphism, and the arrow $\alpha = \ker(\pi_1) \circ \ker(\eta_{K[f]})$ is a normal monomorphism. Since $\pi_2 \circ \alpha= \ker(f) \circ \ker(\eta_{K[f]})$ is a monomorphism, and $\pi_2(\overline{0}_{K[f]}^{\Fc})\cong \overline{0}_{K[f]}^{\Fc}$, we find that $\overline{0}_{K[f]}^{\Fc}$ is normal in $A$ (as a regular image of a normal monomorphism).
\end{proof}
  
\begin{remark}
Notice that, for any normal subobject $K\triangleleft A$ of an object $A$ of $\Bc$, its closure with respect to the reflection $(F,U,\eta^{2},\epsilon^2)$ coincides with its closure with respect to the composite reflection $(F \circ I, H \circ U, \eta, \epsilon)$, so that there is no ambiguity in the notation $\overline{K}^{\Fc}_A$.
\end{remark}

We are now in a position to describe the reflector $F_1  \colon \CExt_{\Bc}(\Ac)  \rightarrow \CExt_{\Fc}(\Ac) $. For this, consider $f \colon A \rightarrow B$ in $\CExt_{\Bc}(\Ac)$, and the following commutative diagram
$$\xymatrix@=20pt{
\overline{0}_{K[f]}^{\Fc}  \ar@{{|>-}->}[d]_-{\ker(\eta_{K[f]})} \ar@{{|>-}->}[drr]^-{\ker(h)}  \\
K[f]  \ar@{{-|>}}[d]_-{\eta_{K[f]}}  \ar@{{|>-}->}[rr]^-{\ker(f)}  && A \ar@{{-|>}}[d]_-h \ar@{{-|>}}[r]^-f  & B\ar@{=}[d] \\
F(K[f])  \ar@{.>}[rr]_-{\alpha}  & & \frac{A}{\overline{0}_{K[f]}^{\Fc}} \ar@{{-|>}}[r]_-{\ttildef}  & B
}$$
where $h$ is the cokernel of the normal monomorphism $\ker(f) \circ \ker(\eta_{K[f]})$ and $\alpha$ is the factorisation of $h \circ \ker(f)$ through $\eta_{K[f]}$.
By applying Lemma \ref{homologique} we see that $\alpha$ is a monomorphism. The arrow $\alpha \circ \eta_{K[f]}$ is then the (regular epimorphism, monomorphism)-factorisation of $h \circ \ker(f)$. From the uniqueness of this factorisation it easily follows that $\alpha = \ker(\tilde{f})$, and $K[\tilde{f}] = F(K[f]) \in \Fc$.
 
  We can  now show that $\ttildef$ is central with respect to $\Bc$, so that $\ttildef$ will be in $\CExt_{\Fc}(\Ac)$ (by Theorem \ref{Caracterisation}). Since $h$ and $\eta_{K[f]}$ are regular epimorphisms, the induced restriction $[K[f],A]_{\Bc} \rightarrow [K[\ttildef],A/\overline{0}_{K[f]}^{\Fc} ]_{\Bc}$ is a regular epimorphism as well. Hence, $[K[\ttildef],A/\overline{0}_{K[f]}^{\Fc} ]_{\Bc}$ is zero since, by assumption, so is $[K[f],A]_{\Bc}$.\\
 We define $F_1(f) = \ttildef$, and verify that $F_1(f)$ has the desired universal property. Let $f': A' \rightarrow B'$ be an extension in $\CExt_{\Fc}(\Ac)$ and $(a, b) \colon f \rightarrow f'$ be an arrow in $\CExt_{\Bc}(\Ac)$ making commutative the right-hand square in the following diagram:
$$ \xymatrix{
\overline{0}_{K[f]}^{\Fc} \ar[r]^{\ker(h)}  \ar[d]_{\hat{a}} &  A \ar[d]_{a} \ar[r]^f& B \ar[d]^{b}\\
\overline{0}_{K[f']}^{\Fc} \ar[r] & {A'} \ar[r]_{f'} & {B'.}
 }$$
 Since $K[f'] \in \Fc$, one has that $\overline{0}_{K[f']}^{\Fc}=0$; from the commutativity of the left-hand square it follows that $a \circ \ker(h) = 0$, and there is then a (necessarily unique) factorisation $c \colon \frac{A}{\overline{0}_{K[f]}^{\Fc}}  \rightarrow A'$ with $c \circ h = a$. The arrow $(c, b) \colon \tilde{f} \rightarrow f'$ is the desired factorisation in the category $\CExt_{\Bc}(\Ac)$.

 \begin{remark}
 As shown in \cite{BG}, any torsion theory $(\mathcal T, \mathcal F)$ in a homological category $\Ac$ determines a semi-left-exact \cite{CHK} reflector $F \colon \Ac \rightarrow \Fc$ to the torsion-free subcategory $\Fc$ of $\Ac$. This means in particular that the corresponding Galois structure is admissible (see \cite{GJ, GR}). An important class of examples of the composite adjunction considered in this section is given by any adjunction \eqref{Composite-adjunction} with $\Fc$ a torsion-free subcategory of $\Bc$ for $(\mathcal T, \mathcal F)$ a hereditary torsion theory and $\Bc$ Birkhoff in $\Ac$. Indeed, it is easy to check that, under our assumptions, the fact that the torsion subcategory $\mathcal T$ is closed in $\mathcal A$ under subobjects implies that the reflector $F \colon \Ac \rightarrow \Bc$ is a protoadditive functor.
\end{remark} 



\section{The generalised Hopf formula}\label{Section-formule}


Before proving the main result of this section---a generalised Hopf formula---we need a few lemmas. We begin by stating a technical result proved in \cite{J4} in a more general context.
\begin{lemma}\label{Cube-Monos} Let $\Ac$ be a homological category. Consider the following cube
$$\xymatrix@=15pt{
K `^ L \ar@{.>}[dd]  \ar@{{ >}->}[dr] \ar@{{ >}->}[rr] & & L\ar[dd] \ar@{{ >}->}[dr] \\
&K\ar[dd] \ar@{{ >}->}[rr] && A\ar@{{-|>}}[dd]^-f \\
U `^V\ar@{{ >}->}[dr] \ar@{{ >}->}[rr] && V\ar@{{ >}->}[dr] \\
&U\ar@{{ >}->}[rr] & &B
}$$
where $U$ and $V$ are subobjects of $B$, $f$ is a regular epimorphism, $K=f^{-1}(U)$ and $L=f^{-1}(V)$.
Then $$U`^ V \cong \frac{K`^L}{K[f]}.$$
\end{lemma}

We continue with the following simple observations.
\begin{lemma}\label{Carre-Zero} 
$(1)$ Let $\Ac$ be any category and $\Bc$ be a reflective subcategory of $\Ac$ with reflector $I\colon \Ac\to \Bc$ and inclusion functor $H\colon \Bc\to \Ac$. If $f$ is an epimorphism such that the unit $\eta_A$ factors through $f$
\[
\xymatrix{
A \ar@{->>}[d]_f \ar[r]^-{\eta_A} & HI(A) \\
B \ar@{.>}[ru]_e}
\]
then the factorisation $e$ is necessarily the unit $\eta_B\colon B\to HI(A)\cong HI(B)$.

$(2)$ In particular, in the case of a pointed category with kernels $\Ac$ and a normal epi-reflective subcategory $\Bc$, if $f\colon A\to B$ is a normal epimorphism such that $K[f] \leq K[\eta_A]$, it follows that induced commutative square
$$\xymatrix{
K[\eta_A] \ar@{.>}[d]_-{\hat f} \ar@{{|>-}->}[r] & A \ar@{{-|>}}[d]^-f\\
K[\eta_B] \ar@{{|>-}->}[r] & {B}
}$$
is a pullback. 
\end{lemma}

\noindent {\bf{Convention.}} \emph{From now, until the end of this section, we shall assume that a composite adjunction \eqref{Composite-adjunction} has been fixed, which satisfies the same conditions as the ones at the beginning of Section \ref{Section-principale}.}

We need one more lemma, which is well known in concrete examples: we state it explicitly for future references.
\begin{lemma}\label{Existence-Ext-FU} If $\Ac$ has enough projectives with respect to $\Ec$, one can construct, for any $B$ in $\Ac$, a weakly universal normal extension of $B$.
\end{lemma}
 
\begin{proof}
Let $B$ be an object of $\Ac$ and $f: P\rightarrow B$ a \emph{projective presentation of B}, i.e. $f \in \Ac \downarrow B$ and $P$ is projective with respect to $\Ec$. Let us show that $F_1 I_1 (f)$, which is a normal extension of $B$, is also weakly universal.
If $(E,p)$ is a normal extension of $B$, since $P$ is projective with respect to $\Ec$, there exists an arrow $\alpha:P \rightarrow E$ such that $f= p \circ \alpha$. By the universal property of $F_1 I_1 (f)$, one gets the desired factorisation of $F_1 I_1 (f)$:
$$\xymatrix{
P \ar@{{-|>}}[rr]^-f \ar@{{-|>}}[d] \ar@{.>}[rdd]^-(.7){\alpha}&&B\\
\tilde{P} =\frac{{P}/{[K[f], P]_{\Bc}}}{\overline{0}_{\frac{K[f]}{[K[f], P]_{\Bc}} }^{\Fc}} \ar@{.>}[dr]_-{\beta} \ar@{{-|>}}[urr]_-{F_1I_1(f)}&&\\
&E \ar@{{-|>}}[ruu]_-p&
}$$
\end{proof}
Thanks to this observation, one can compute  $\pi_1(B)$ for any $B$ in $\Ac$, starting from any projective presentation $f: P\rightarrow B$ of $B$. Indeed, $F_1I_1(f)$ is a weakly universal extension of $B$, as shown in the previous lemma, and
$$\pi_1(B)=\mathsf{Gal}(\tilde{P}, F_1I_1(f), 0) \cong K[F_1I_1(f)]`^K[\eta_{\tilde{P}}].$$
The formula appearing in the next theorem entirely describes $\pi_1(B)$, in terms of the closure operator associated with the composite reflection, without any reference to $F_1I_1(f)$.

\begin{theorem}\label{Formule} Let $B$ be an object of $\Ac$ and $f:P \rightarrow B$ a projective presentation of $B$. Then
$$\pi_1(B)\cong \dfrac{K[f] `^ \overline{([P,P]_{\Bc})}^{\Fc}_P }{\overline{([K[f],P]_{\Bc})}^{\Fc}_{K[f]}}.$$
\end{theorem}

\begin{proof}
One can first remark that all the faces in the following cubes are pullbacks (we denote here $\check f=I_1(f)$ and $\tilde f = F_1(\check f)= F_1 I_1 (f)$ and write $g\colon P\to\check{P}$ and $h\colon \check{P}\to\tilde{P}$ for the canonical quotients).
$$\xymatrix@=15pt{
K[f] `^  \overline{0}^{\Fc}_{P}\ar@{.>}[dd]  \ar[dr] \ar[rr] & & \overline{0}^{\Fc}_{P} \ar[dd] \ar[dr] \\
&K[f]\ar[dd] \ar[rr] && P\ar@{{-|>}}[dd]^-g\\
K[\check f] `^\overline{0}^{\Fc}_{\check P}\ar[dr] \ar[rr] && \overline{0}^{\Fc}_{\check P}\ar[dr] \\
&K[\check f] \ar[rr] & &P/[K[f],P]_{\Bc}=\check{P}}$$
$$\xymatrix@=15pt{
 K[\check f] `^\overline{0}^{\Fc}_{\check P}\ar@{.>}[dd]  \ar[dr] \ar[rr] & &  \overline{0}^{\Fc}_{\check P}\ar[dd] \ar[dr] \\
&K[\check f] \ar[dd] \ar[rr] && \check {P}\ar@{{-|>}}[dd]^-h\\
K[\tilde f] `^\overline{0}^{\Fc}_{\tilde P}\ar[dr] \ar[rr] && \overline{0}^{\Fc}_{\tilde P}\ar[dr] \\
&K[\tilde f] \ar[rr] & &\check{P}/\overline{0}^{\Fc}_{K[\check f]} =\tilde{P}}$$
For the first cube, one sees that $K[g]=[K[f],P]_{\Bc} \leq [P,P]_{\Bc} \leq \overline{0}^{\Fc}_{P}$ which entails that its right hand face is a pullback, by the second part of Lemma \ref{Carre-Zero}. The fact that its front face is a pullback follows from Lemma \ref{homologique}, so that all the other faces are pullbacks as well.
For the second cube, one follows the same lines: one just remarks that $K[h]\leq  \overline{0}^{\Fc}_{\check P}$, since $\eta_{\check P} \circ \ker(h)=FI(\ker (\check{f}))\circ \eta_{K[\check{f}]}\circ \ker (\eta_{K[\check{f}]})=0$.\\
 The cuboid made of the two cubes above is of the same type of the one in Lemma \ref{Cube-Monos}, and one then finds:
 $$\pi_1(B)\cong \frac{K[f] `^  \overline{0}^{\Fc}_{P}}{K[h \circ g]}.$$
 We now rewrite the terms on the right side. One clearly has that $$\overline{0}^{\Fc}_{P}=\overline{(\overline{0}^{\Bc}_{P})}^{\Fc}_{P}= \overline{([P,P]_{\Bc})}^{\Fc}_P.$$ Furthermore, by looking at the diagram
 $$
 \xymatrix{
 \hat{g}^{-1}(\overline{0}^{\Fc}_{K[\check f]}) \ar[r] \ar[d] & K[f] \ar[r]^-{\ker(f)} \ar@{{-|>}}[d]_-{\hat g}& P\ar@{{-|>}}[d]^-g\\
\overline{0}_{K[\check f]}^{\Fc} \ar@{{|>-}->}[r]_-{{\ker}({\eta_{K[\check f]})}} \ar[d]& K[\check f] \ar@{{|>-}->}[r]& \check P\ar@{{-|>}}[d]^-h\\
 0\ar[rr]& & \tilde P 
 }$$
 in which all rectangles are pullbacks, we see that there is an isomorphism between the  domains $K[h\circ g]$ and $\hat{g}^{-1}(\overline{0}_{K[\check f]}^{\Fc})$ of the normal monomorphisms $\ker(h\circ g)$ and ${\hat g}^{-1}({\ker}({\eta_{K[\check f]})})$, respectively. Since $\hat g$ is a regular epimorphism, and the closure operator corresponding to the regular epi-reflection $F\circ I \colon \Ac \rightarrow \Fc$ is \emph{homological} (so that axiom $(5)$ holds), one has the following equalities:
 $$\begin{array}{rcl}
\hat{g}^{-1}(\overline{0}^{\Fc}_{K[\check f]}) &=& \overline{(\hat{g}^{-1}(0))}^{\Fc}_{K[f]}\\
 &=& \overline{(K[\hat g])}^{\Fc}_{K[f]}\\
  &\stackrel{(*)}{=}& \overline{(K [g])}^{\Fc}_{K[f]}\\
  &=& \overline{([K[f],P]_{\Bc})}^{\Fc}_{K[f]}.
  \end{array}$$
  $(*)$  Here, by abuse of notation, $K[g]$ denotes the (domain of the) kernel  \\
  $\xymatrix{K[g] \ar@{{|>-}->}[r] & K[f]}$ of $\hat g$. More precisely, the kernel of $h\circ g$ is the normal monomorphism $$\xymatrix{\overline{([K[f],P]_{\Bc})}^{\Fc}_{K[f]} \ar[rr]^-{\ker(f)\circ \overline{\ker(\hat g)}} && P}.$$
  \end{proof}



\section{Examples}

\noindent {\bf Groups with coefficients in torsion-free abelian groups.}

\noindent We consider, as a first example, the adjunction
$$\xymatrix@=30pt{
\Ab_{t.f.} \ar@/_/[r]_-U   \ar@{}[r]|-{\perp} & \Ab \ar@/_/[r]_-H \ar@{}[r]|-{\perp} \ar@/ _/[l]_-F & \Gp \ar@/ _/[l]_-{ab}}$$
where $\Gp$ is the category of groups, $\Ab$ the category of abelian groups and $\Ab_{t.f.}$ the category of torsion-free abelian groups (this is the torsion-free part of the classical torsion theory $(\Ab_{t.},\Ab_{t.f.})$ where $\Ab_{t.}$ is the category of torsion abelian groups). This composite adjunction is an instance of $(\ref{Composite-adjunction})$, since the reflector $F  \colon \Ab \rightarrow \Ab_{t.f.}$ is an additive functor, thus a protoadditive functor. Note that the kernel of the $A$-component of the unit $\eta$ of this adjunction at an abelian group $A$ is given by 
$$K[\eta_A]=\overline{0}^{\Ab_{t.f.}}_A=\{a \in A \text{  $|$ }  \exists n \in \mathbb{N}_0: a^n=1\}.$$ 
Now, when $K$ is a normal subgroup of a group $A$, with quotient map $q_K\colon A\to A/K$, such that $K \geq \overline{0}^{\Ab}_{A}=[A,A]_{\Ab} =[A,A]$ (here the commutator is the group-theoretic one, thus the quotient group $A/K$ is abelian), one has that
 $$\begin{array}{rcl}
 \overline{K}^{\Ab_{t.f.}}_{A} &=&q_K^{-1}(\overline{0}^{\Ab_{t.f.}}_{A/K})\\
 &=& \{a \in A \text{  $|$ }  \exists n \in \mathbb{N}_0: (Ka)^n=K\}\\
  &= &\{ a \in A\text{ $|$  }\exists n \in \mathbb{N}_0: a^n \in K\}. \end{array}$$
Consider then any free presentation $$\xymatrix{
0\ar[r]&K\ar@{{|>-}->}[r]  & P\ar@{{-|>}}[r] &B\ar[r]&0
}$$
of a group $B$.
In order to compute the generalised Hopf Formula in Theorem \ref{Formule}, we first observe that $\overline{0}^{\Ab}_{P}=[P,P]_{\Ab} =[P,P]$ and $\overline{0}^{\Ab}_{K}=[K,K] \leq [K,P]$, so that the description of the closure with respect to $\Ab_{t.f.} $ given above applies to $[P,P]$ and to $[K,P]$. 
Consequently, the fundamental group can be computed as follows:
$$\begin{array}{rcl}
\pi_1(B) &\stackrel{\ref{Formule}}{\cong}&\dfrac{K`^ \overline{([P,P])}^{\mathsf{Ab}_{t.f.}}_P }{\overline{([K,P])}^{\Ab_{t.f.}}_{K}}\\
&\cong &\dfrac{K `^\{ p \in P \text{  $|$  } \exists n \in \mathbb{N}_0: p^n \in [P,P]\}}{\{ p \in K \text{  $|$  } \exists n \in \mathbb{N}_0: p^n \in [K,P]\}}\\
&=& \dfrac{\{ p \in K  \text{  $|$  } \exists n \in \mathbb{N}_0: p^n \in [P,P]\}}{\{ p \in K \text{  $|$  } \exists n \in \mathbb{N}_0: p^n \in [K,P]\}}. \end{array}$$
\vspace{3mm} 

\noindent {\bf Rings with coefficients in reduced commutative rings.} \\
Let $\mathsf{Rng}$ be the semi-abelian category of (not necessarily unitary) rings and $\mathsf{CRng}$ its subvariety of (not necessarily unitary) commutative rings: we denote the corresponding reflection
$$\xymatrix@=30pt{
  \mathsf{CRng} \ar@/_/[r]_-H \ar@{}[r]|-{\perp}  & \mathsf{Rng.} \ar@/ _/[l]_-{I}}$$
Let then $\mathsf{RedCRng}$ be the category of reduced commutative rings, which is the
full replete subcategory of $\mathsf{CRng}$ whose objects have no (non-zero) nilpotent element. In other words, the rings in $\mathsf{RedCRng}$ satisfy all implications of the form $ x^n=0 \Rightarrow x=0$ (with $n\geq 1)$. 
This yields a regular epi-reflective subcategory 
$$\xymatrix@=30pt{
\mathsf{RedCRng} \ar@/_/[r]_-U   \ar@{}[r]|-{\perp} & \mathsf{CRng} \ar@/ _/[l]_-F
}$$
 where $\mathsf{RedCRng}$ is the torsion-free part of a hereditary torsion theory  $$(\mathsf{CRng_{Nil}}, \mathsf{RedCRng})$$ in $\mathsf{CRng}$, whose torsion part is the subcategory $\mathsf{CRng_{Nil}}$ of nilpotent commutative rings (see \cite{CDT}, for instance). The homological closure operator associated with this last reflection can be described explicitely, and it actually gives the well known notion of \emph{radical} of an ideal. Indeed, for any ideal $I$ of a commutative ring $A$, its closure in $A$ is its radical in $A$, written $\sqrt{I}_{(A)}$:
$$\begin{array}{rcl}
  \overline{I}^{\mathsf{RedCRng}}_{A} &=& q_I^{-1}(\overline{0}_{A/I}^{\mathsf{RedCRng}})\\
 &=& \{a \in A \text{  $|$ }  \exists n \in \mathbb{N}_0: (I+a)^n=I\}\\
  &= &\{ a \in A\text{ $|$  }\exists n \in \mathbb{N}_0: a^n \in I\}\\
  &=& \sqrt{I}_{(A)}. 
  \end{array}$$
One can then consider the following composite adjunction
$$\xymatrix@=30pt{
\mathsf{RedCRng} \ar@/_/[r]_-U   \ar@{}[r]|-{\perp} &  \mathsf{CRng} \ar@/_/[r]_-H \ar@{}[r]|-{\perp} \ar@/ _/[l]_-F & \mathsf{Rng} \ar@/ _/[l]_-{I}}$$
where the reflector $F \colon \mathsf{CRng} \rightarrow \mathsf{RedCRng}$ is indeed protoadditive, as one can easily see by using the fact that the torsion theory $(\mathsf{CRng_{Nil}}, \mathsf{RedCRng})$ is hereditary. This adjunction is then another example of our composite adjunction \eqref{Composite-adjunction}. Given a free presentation of a ring $B$
$$\xymatrix{
0\ar[r]&K\ar@{{|>-}->}[r]  & P\ar@{{-|>}}[r] &B\ar[r]&0
}$$
the generalised Hopf formula for $\pi_1(B)$ here becomes:
$$\pi_1(B)\cong \dfrac{K `^ \sqrt{[P,P]_{\mathsf{CRng}}}_{(P)}} {\sqrt{[K,P]_{\mathsf{CRng}}}_{(K)}}$$
where $[P,P]_{\mathsf{CRng}}=( \{ pp'-p'p \text{ $|$ } p,p' \in P\} )$ is the ideal of $P$ generated by all the elements of the form $pp'-p'p$ for  $p,p' \in P$ and, similarly, $$[K,P]_{\mathsf{CRng}}=( \{ pk-kp \text{ $|$ } k\in K,\text{  } p \in P\} )$$ is the ideal of $P$ generated by all elements of the form $pk-kp$ for $k \in K$, $p \in P$.
\vspace{3mm} 

\noindent {\bf Compact groups with coefficients in abelian profinite groups.}\\
We now consider the following composite adjunction:
$$\xymatrix@=30pt{
\Ab(\mathsf{Prof}) \ar@/_/[r]_-U   \ar@{}[r]|-{\perp} & \Ab(\mathsf{HComp}) \ar@/_/[r]_-H \ar@{}[r]|-{\perp} \ar@/ _/[l]_-F & \Gp(\mathsf{HComp}) \ar@/ _/[l]_-{\mathsf{ab}}.}$$
Here $\mathsf{Grp} (\mathsf{HComp})$ is the semi-abelian category (see \cite{BC}) of compact (Hausdorff) groups, $\mathsf{Ab} (\mathsf{HComp})$ is its Birkhoff subcategory of abelian compact groups, and $\mathsf{Ab} (\mathsf{Prof})$ its Birkhoff subcategory of profinite abelian groups.  As usual, $U$ and $H$ are full inclusions, and we write $\mathsf{ab}$ for the left adjoint of $H$, which sends a compact group $G$ to the quotient $G/\overline{[G,G]}^{\mathsf{top}}$ of $G$ by the normal subgroup ${\overline{[G,G]}}^{\mathsf{top}}$, the \emph{topological closure} in $G$ of the derived subgroup $[G, G]$ of $G$. The left adjoint of $U$, here denoted by $F$, sends an abelian compact group $A$ to the quotient $A/A_0$ of $A$ by the connected component $A_0$ of the neutral element $0$ of $A$. Since the category $ \Ab(\mathsf{HComp})$ is abelian, the reflector $F \colon  \Ab(\mathsf{HComp}) \rightarrow \Ab(\mathsf{Prof})$ is necessarily (proto)additive, so that the composite adjunction above is another special instance of the adjunction $(\ref{Composite-adjunction})$. Note that the category $\Gp(\mathsf{HComp})$ has enough regular projectives, since it is monadic over the category of sets \cite{Manes}. Hence, Theorem \ref{Formule} gives us a characterisation of the fundamental group with respect to this adjunction. As in the previous examples, we can give an explicit description of the homological closure operator in this situation. In order to do this, let us consider also the Birkhoff subcategory $\Gp(\mathsf{Prof})$ of   $\mathsf{Grp} (\mathsf{HComp})$   of profinite groups. The reflection of a compact group $G$ in this subcategory is given by the quotient $G/G_0$, with $G_0$ the connected component of the neutral element of $G$. Note that the internal product $K\cdot L$ of two closed normal subgroups of a compact group $G$ (normal subobjects in the category $\mathsf{Grp} (\mathsf{HComp})$) is necessarily closed, so that it is the supremum $K\vee L$ in the lattice of (normal) subobjects of $A$. Using that $\Ab(\mathsf{Prof}) \leq \Gp(\mathsf{Prof})$, as well as Lemma \ref{Fermeture} (2), we find, for any normal subobject $K$ of $G$ such that  $K\geq \overline{[G,G]}^{\mathsf{top}}$---which implies that $G/K\cdot G_0\in\Ab(\mathsf{Prof})$---that
\[
 \overline{(K)}^{\Ab(\mathsf{Prof})}_G\leq  \overline{(K\cdot G_0)}^{\Ab(\mathsf{Prof})}_G= K\cdot G_0= \overline{(K)}^{\Gp(\mathsf{Prof})}_G\leq  \overline{(K)}^{\Ab(\mathsf{Prof})}_G
\]
and we see that the inequalities are, in fact, equalities. Hence, for any compact group $B$, and any projective presentation $$\xymatrix{ 0 \ar[r] & K \ar[r] & P \ar[r] & B \ar[r] & 0.}$$ the characterisation of the fundamental group of $B$ from Theorem \ref{Formule} becomes
$$ \pi_1 (B) {\cong} \frac{K \wedge (\overline{[P,P]}^{\mathsf{top}} \cdot P_0 )}{\overline{[K , P]}^{\mathsf{top}} \cdot K_0 },$$
 It is not difficult to extend this result to the context of compact semi-abelian algebras by applying the methods of Borceux and Clementino in \cite{BC}.

\vspace{3mm} 

\noindent {\bf Simplicial loops with coefficients in groups.} \\
For this last example, we consider a semi-abelian category $\Ac$ with Birkhoff subcategory $\Bc$ and write, as usual, $H$ for the inclusion functor and $I$ for its left adjoint. We denote by $\Sc (\Ac)$ and $\Sc (\Bc)$ the categories of simplicial objects in $\Ac$ and $\Bc$, respectively. Now consider the composite adjunction
$$\xymatrix@=30pt{
\Bc \ar@/_/[r]_-{D}   \ar@{}[r]|-{\perp} & \Sc(\Bc) \ar@/_/[r]_-{\Sc(H)} \ar@{}[r]|-{\perp} \ar@/ _/[l]_-{\pi_0} & \Sc(\Ac) \ar@/ _/[l]_-{\mathsf{\Sc}(I)}.}$$
Here the functors $\Sc(H)$ and $\Sc(I)$ are induced by $H$ and $I$, respectively, $\pi_0$ is the ``connected components" functor and $D$ its right adjoint, which maps an object of $\Bc$ to the associated discrete simplicial object. $\Sc(\Bc)$ is, of course, a Birkhoff subcategory of $\Sc(\Ac)$, and $\Bc$ a Birkhoff subcategory of $\Sc(\Bc)$.  Furthermore, one can prove that $\pi_0$ is a protoadditive functor by considering with any split short exact sequence of simplicial objects in $\Ac$ (the final part of) the induced long exact sequence (see Corollary $5.7$  in \cite{EVD}) and by taking into account that, in a pointed protomodular category, a morphism is a monomorphism if its kernel is zero. Notice also that $\Bc$ is closed under extensions in $\Sc(\Bc)$ by the Short Five Lemma; one can then check that the Corollary in \cite{JT} applies here, so that $\Bc$ is a torsion-free subcategory of $\Sc(\Bc)$. In this example the torsion subcategory consists of the simplicial objects that are connected.  By using similar arguments as in the previous example, we can obtain characterisations of the fundamental group for different choices of $\Ac$ and $\Bc$. 

For instance, let $\Ac$ be the variety $\Loop$ of loops: recall that its algebraic theory has three binary operations $\cdot, \backslash, /$, called multiplication, left division and right division, respectively, and a unique constant $1$ satisfying the identities $$y= x \cdot (x\backslash y), \quad y= x \backslash (x \cdot y), \quad x=(x /y) \cdot y, \quad x=(x \cdot y) / y, \quad x\cdot 1 =x= 1 \cdot x.$$ This variety is semi-abelian, as shown in \cite{BC}, and it contains the variety $\Bc= \Gp$ of groups as a subvariety (since a loop is a group if and only if the multiplication is associative). We know from \cite{EVDAU} that, for any surjective homomorphism of loops $f\colon A\to B$, with kernel $K[f]$, the corresponding relative commutator $[K[f],A]_{\Gp}$ is the ``associator" $[K[f],A,A]$. Thanks to this result and to Theorem \ref{Formule} we find that, for any projective presentation   
$$\xymatrix{ 0 \ar[r] & K \ar[r] & P \ar[r] & B \ar[r] & 0}$$
of a simplicial loop $B$, there is an isomorphism
$$ \pi_1 (B) {\cong} \frac{K \wedge ([P,P,P] \vee P_0  )}{[K , P,P] \vee K_0 },$$
with $P_0$ and $K_0$ the connected simplicial objects determined by the connected components of $0$ in $P$ and in $K$, respectively, and the ``associator'' of simplicial loops is defined degreewise.

Note that the category $\Sc(\Loop)$ has enough projectives as a consequence of the following two facts. On the one hand, the category $\Sc(\Set)$ of simplicial sets has enough projectives (as any category of presheaves---see, for instance, Exercise IV.$15$ in \cite{MM}); on the other hand, for any monadic functor $F\colon \Xc\to\Yc$ that preserves regular epimorphisms, one has that $\Xc$ has enough projectives as soon as so has $\Yc$ (as explained in the proof of Proposition $3.2$ in \cite{EG}), and we can apply this result, in particular, in the case of the forgetful functor $\Sc(\Loop)\to \Sc(\Set)$.

\end{document}